\documentclass{article}
\usepackage{amsfonts}
\usepackage{amsmath}
\usepackage{dsfont}
\usepackage{graphicx}
\usepackage{amssymb}
\usepackage{caption}
\usepackage[lined,boxed,commentsnumbered]{algorithm2e}
\usepackage{geometry}
\RequirePackage[OT1]{fontenc}
\newtheorem{theorem}{Theorem}

\newtheorem{lemma}[theorem]{Lemma}

\newtheorem{remark}[theorem]{Remark}

\newenvironment{proof}[1][Proof]{\noindent\textbf{#1.} }{\ \rule{0.5em}{0.5em}}

\begin{document}

\title{Approximation of a multivariate conditional density.} 
\author{Virgile Caron\\ TelecomParisTech\\virgile.caron@telecom-paristech.fr}
\maketitle
\begin{abstract}
This paper extends the result of Broniatowski and Caron (2013) pertaining to the asymptotic distribution of a random walk conditioned on its final value as the number of summands increase. We consider multivariate light-tailed random walk and present a sharp approximation of long runs conditioned by an average of a function of its summands as their number tends to infinity.
\end{abstract}

\section{Introduction} \label{sec:intro_chap4}

Consider $\textbf{X}_{1}^{n}:=(\textbf{X}_{1},...,\textbf{X}_{n})$ a set of $n$ independent copies of a $d$-dimensional random vector $\mathbf{X}:=^{t}(\textbf{X}^{(1)},...,\textbf{X}^{(d)})$ with density $p_{\mathbf{X}}$ on $\mathbb{R}^{d}.$ The integer $d$ is assumed to be greater than one. The superscript $(j)$ pertains to the coordinate of a vector and the subscript $i$ pertains to replications. 

We consider the approximation of the density of the vector $\mathbf{X}_{1}^{k}$ on $\left(\mathbb{R}^{d}\right)^{k}$, when the conditioning event writes 

\begin{eqnarray}  \label{cond:surS}
\left(\mathbf{S}_{1,n}:=\mathbf{X}_{1}+...+\mathbf{X}_{n}=na_{n}\right)
\end{eqnarray}
with  $a_{n}:=^{t}(a_{n}^{(1)},...,a_{n}^{(d)})$ a convergent sequence and the integer value sequence $k:=k_{n}$ is such that 
\begin{align} \label{cond_dur_k}
\lim_{n \to \infty} \frac{k}{n}=1 \tag{K1} \\
\lim_{n \to \infty} (n-k)=+\infty \tag{K2}
\end{align}

Therefore we may consider the asymptotic behavior of the density of the random walk on long runs. We also consider a more general case when $\mathbf{S}_{1,n}$ is substituted by 
\begin{eqnarray}\label{cond:surU}
\mathbf{U}_{1,n}:=u(\mathbf{X}_{1})+...+u(\mathbf{X}_{n})
\end{eqnarray} 
for some measurable function $u$ defined from $\mathbb{R}^{d}$ to $\mathbb{R}^{s}$ and when the conditionning event writes $\left(\mathbf{U}_{1,n}:=u_{1,n}\right)$ where $u_{1,n}/n$ converges as $n$ tends to $\infty.$ The integer $s$ is assumed to be larger than one and lower than $n.$ The last condition on $s$ forbid to have a zero-one conditional event.

In \cite{BroniatowskiCaron2013}, the approximation of the density of the real vector $\mathbf{X}_{1}^{k}$ conditioned on $\left(\mathbf{U}_{1,n}:=u_{1,n}\right)$ is obtained either in probability under the conditional distribution of the random walk either in total variation norm between measures. We provide an extension of this results since we consider random vector on $\mathbb{R}^{d}$ and a function $u$ which is defined from $\mathbb{R}^{d}$ to $\mathbb{R}^{s}.$

In \cite{BroniatowskiCaron2013}, two mains advances over the literature were made. First, $k_{n}$ is allows to be on the same order than $n$ at the condition than $n-k_{n}$ tends to infinity when $n$ tends to infinity. Secondly, the conditioning event can be in the range of the central limit theorem, allowing to conditional inference (see \cite{BroniatowskiCaron2012Exp}) or in large deviation, allowing to estimate rare event probability (see \cite{BroniatowskiCaron2011IS}). 

The paper is organized as follows. Section \ref{sec:hyp_chap4} presents the main notations and general hypothesis. In Section \ref{sec:prop_princ_chap4}, we expose the approximation scheme for the conditional density of $\mathbf{X}_{1}^{k}$ when the conditioning event writes (\ref{cond:surS}). This approximation is extended to the case when the conditionning event writes $\left(\mathbf{U}_{1,n}=u_{1,n}\right)$ in Section \ref{sec:extension_chap4}. In Section \ref{sec:choi_k_chap4}, we discuss the value of $k$ for which the approximation described is valid. Some technicals lemmas are left to the Appendix while the main step of the proof are in the core of the paper.

\section{Notations and hypotheses} \label{sec:hyp_chap4}

\subsection{General Notations}

In this section, the conditionning event writes
\begin{eqnarray*}
\left(\mathbf{S}_{1,n}:=\mathbf{X}_{1} +...+\mathbf{X}_{n}=na_{n}\right).
\end{eqnarray*}

For sake of clarity, this event can be written as follows
\begin{eqnarray}\label{def:event_to_estime}
P_{n}=P\left[\bigcap_{j=1}^{d}\left\{\sum_{i=1}^{n}\mathbf{X}_{i}^{(1)},...,\mathbf{X}_{i}^{(d)}=na_{n}^{(j)}\right\}\right]
\end{eqnarray}
where, for $j\in\{1,...s\}$, $a_{n}^{(j)}\in\mathbb{R}.$

We assume that the characteristic function of $\mathbf{X}$ is in $L^{r}$ for some $r\geq{1}.$ This hypothesis is necessary to perform an Edgeworth expansion. We additionally assume that $\mathbf{X}$ satisfies the Cramer condition, i.e. $\mathbf{X}$ has a finite moment generating function in a non void neighborhood of $\underline{0}$
\begin{eqnarray} \label{def:fgm_r_d}
\Phi(t):=E[\exp\left(<t,\textbf{X}>\right)]<\infty, \textsc{\ \ \ } t\in V(\underline{0})\subset\mathbb{R}^{d}
\end{eqnarray} 
where $V(\underline{0})$ is a neighborhood of $\underline{0}$, which is the vector of $\mathbb{R}^{d}$ with all coordinates equal to zero and $<.,.>$ is the standard inner product in $\mathbb{R}^{d}.$ In other words, we consider only light tailed random vector.

Denote  
\begin{eqnarray}
m(t):= \nabla \log(\Phi(t)), \textsc{\ \ \ } t\in V(\underline{0})\subset\mathbb{R}^{d}
\end{eqnarray}
and
\begin{eqnarray} \label{def:cov_matrix}
\kappa(t):= {}^{t}\nabla \nabla m(t), \textsc{\ \ \ } t \in V(\underline{0})\subset\mathbb{R}^{d}.
\end{eqnarray}

The value of $m(t_{\alpha}):= {}^{t}\nabla \log\Phi(t_{\alpha})$  and
$\kappa(t_{\alpha}):= {}^{t}\nabla \nabla \log\Phi(t_{\alpha})$ are respectively the expected value and the covariance matrix of the \textit{tilted} density defined by
\begin{eqnarray} \label{def:pi_R_d_chap5}
\pi^{\alpha}(x) :=\frac{\exp <t,x>}{\Phi(t)} p(x)
\end{eqnarray}
where $t$ is the only solution of $m(t)=\alpha$ when $\alpha$ belong to the support of $\mathbf{X}.$ Conditions on $\Phi(t)$ which ensure existence and uniqueness of $t$ are referred to \textit{steepness properties}; we refer to \cite{BarndorffNielsen1978_chap_multi}, p.153 and followings for all properties of moment generating functions used in this paper. Denote $\Pi^{\alpha}$ the probability measure with density $\pi^{\alpha}.$

The conditional density of $\mathbf{X}_{1}^{k}$ on $\left(\mathbb{R}^{d}\right)^{k}$ conditioned by $\left(\mathbf{S}_{1,n}=na_{n}\right)$ will be denoted $p_{na_{n}}.$ For a generic random vector $\mathbf{Z}$ with density $p$, $p\left( \mathbf{Z}=z\right)$ denotes the value of $p$ at point $z.$ Therefore,

\begin{eqnarray}
p_{na_{n}}(Y_{1}^{k}):=p(\textbf{X}_{1}^{k}=Y_{1}^{k}|\mathbf{S}_{1,n}=na_{n})
\end{eqnarray}

The $d$-dimensional normal density with expected value $\mu$ and covariance matrix $\kappa$ at $x$ is denoted $\mathfrak{n}_{d}(x;\mu,\kappa).$ When $\mu=0$ and $\kappa=I_{d}$, we denote $\mathfrak{n}_{d}(x):=\mathfrak{n}_{d}(x;0,I_{d}).$

\subsection{Specific notations}

We introduce some very specific notations in this section. Most of the notations are presented by \cite{BardnoffCox1990} (see p.130 and followings) or in \cite{McCullagh1984}. This notations are standard when dealing with multidimensional Edgeworth expansion, as proved by the great deal of paper using them (see \cite{PaceSalvan1992} on conditional cumulants or \cite{PaceSalvan2004} on likelihood expansions). 

The subscript $j,k,l,...$ with or without suffix are integers between $1$ and $d.$ Furthermore, we adopt the Einstein convention, i.e. if a subscript is up and down in a same expression, then the summation is on the entire space of possible value for this subscript. For example, $a^{j}b_{j}=\sum_{j=1}^{d}a^{j}b_{j}$ or $a^{jlm}b_{jr}=\sum_{j=1}^{d}a^{jlm}b_{jr}.$
 
\subsubsection{Moments and cumulants} \label{sec:moments}

Denote $\kappa^{j,l}$ the generic term of the covariance matrix defined in (\ref{def:cov_matrix}) and $\kappa_{j,l}$ the generic term of the inverse matrix. More generally, the joint moment and the joint cumulant of $(X^{(i_{1})}...X^{(i_{\nu})})$ are denoted respectively

\begin{eqnarray}
\kappa^{i_{1}...i_{\nu}}=E[X^{(i_{1})}...X^{(i_{\nu})}] \\
\kappa^{i_{1},...,i_{\nu}}=K[X^{(i_{1})}...X^{(i_{\nu})}]
\end{eqnarray}

The symbol $[n]$ will indicate a sum of $n$ terms determined by a permutation of the subscripts. For example,
\begin{eqnarray*}
\kappa^{j,l}x_{m}[3]=\kappa^{j,l}x_{m}+\kappa^{j,m}x_{l}+ \kappa^{m,l}x_{j}.
\end{eqnarray*}
Finally, using the summation convention, denote
\begin{eqnarray}
\kappa_{i_{1},...,i_{\nu}}=
\kappa_{i_{1},j_{1}}...\kappa_{i_{\nu},j_{\nu}}\kappa^{j_{1},...,j_{\nu}}
\end{eqnarray}
where $\kappa_{j,l}$ is the generic term of the inverse matrix.

To simplify some notation, we also use the \textit{index} notation (see p.132 of \cite{BardnoffCox1990}). For a set $I=\{i_{1},...,i_{\nu}\}$, we sometimes rewrite the joint moment or the joint cumulant as  
\begin{eqnarray}
\kappa_{1}^{I}:=\kappa_{1}^{i_{1}...i_{\nu}}=E[X^{(i_{1})}...X^{(i_{\nu})}]\\
\kappa_{0}^{I}:=\kappa_{0}^{i_{1},...,i_{\nu}}=K[X^{(i_{1})}...X^{(i_{\nu})}]
\end{eqnarray}

\subsubsection{Tensor Hermite polynomials} \label{sec:herm}

In the main proof, we use an Edgeworth expansion with the associated Hermite polynomials. In $\mathbb{R}^{d}$ the formula giving the tensorial Hermite polynomial $h_{i_{1}...i_{k}}$ (p.150) associated with the $d$-dimensional standard normal density can be obtain using

\begin{eqnarray} \label{def:Hrd}
\mathfrak{n}_{d}(x)h_{i_{1}...i_{k}}(x)=
(-1)^{k}\partial_{i_{1}}...\partial_{i_{1}}\mathfrak{n}_{d}(x)
\end{eqnarray}

where $\partial_{i_{1}}=\frac{\partial}{\partial x^{(i_{1})}}$ and $x^{(i_{1})}$ is the $i_{1}$th coordinate of $x.$

Denote $x_{j}=\kappa_{j,l}x^{(l)}.$ Then the polynomial $h$ used in the multidimensional Edgeworth expansion are

\begin{eqnarray} 
h_{jlm}(x):=x_{j}x_{l}x_{m}-\kappa_{j,l}x_{m}[3]
\end{eqnarray}
\begin{eqnarray}
h_{jlmq}(x):=x_{j}x_{l}x_{m}x_{q}-\kappa_{j,l}x_{m}x_{q}[6]+\kappa_{j,l}\kappa_{m,q}[3]
\end{eqnarray}
\begin{eqnarray}
h_{jlmqrs}(x):=x_{j}x_{l}x_{m}x_{q}x_{r}x_{s}-\kappa_{j,l}x_{m}x_{q}x_{r}x_{s}[15] \notag \\ +\kappa_{j,l}\kappa_{m,q}x_{r}x_{s}[45]-\kappa_{j,l}\kappa_{m,q}\kappa_{r,s}[15]
\end{eqnarray}

\subsubsection{Egeworth expansion} \label{sec:edg}

As we know, the Edgeworth expansion in $\mathbb{R}^{d}$ for the sum of i.i.d. random vectors is based on two ingredrients, first the Hermite polynomial defined in (\ref{def:Hrd}) and the Taylor expansion of the moment generating function. Using the Einstein convention summation, it holds
\begin{eqnarray} \label{def:edgrd}
p_{\frac{\overline{\mathbf{S}_{1,n}}}{\sqrt{n}}}(x):=\mathfrak{n}_{d}(x)(1+\frac{\kappa^{j,l,m}}{6\sqrt{n}}h_{jlm}(x)
+\frac{\kappa^{j,l,m,q}}{24n}h_{jlmq}(x)+
\frac{\kappa^{j,l,m}\kappa^{q,r,s}}{72n}h_{jlmqrs}(x))
\end{eqnarray}
\begin{center}
$+O(\frac{1}{n^{3/2}})$
\end{center}
where $\overline{\mathbf{S}_{1,n}}=\Sigma^{-1/2}\left(\mathbf{S}_{1,n}-n\mu\right).$

Define
\begin{eqnarray} \label{def:Q3}
Q_{3}(x):=\frac{\kappa^{j,l,m}}{6}h_{jlm}(x)
\end{eqnarray}
and
\begin{eqnarray} \label{def:Q4}
Q_{4}(x):=\frac{\kappa^{j,l,m,q}}{24}h_{jlmq}(x)+\frac{\kappa^{j,l,m}\kappa^{q,r,s}}{72}h_{jlmqrs}(x).
\end{eqnarray}

Then, (\ref{def:edgrd}) can be written
\begin{eqnarray} \label{def:edgrd2}
p_{\frac{\overline{\mathbf{S}_{1}^{n}}}{\sqrt{n}}}(x):=\mathfrak{n}_{d}(x)(1+\frac{1}{\sqrt{n}}Q_{3}(x)
+\frac{1}{n}Q_{4}(x))+O(\frac{1}{n^{3/2}}).
\end{eqnarray}

\section{Multivariate random walk conditioned on their sum.} \label{sec:prop_princ_chap4}

Let $\epsilon_{n}$ be a positive sequence such as 
\begin{equation} \label{cond:A1}
\lim_{n \to \infty} \epsilon_{n}^{2}(n-k)=\infty \tag{E1}
\end{equation}
\begin{equation} \label{cond:A2}
\lim_{n \to \infty} \epsilon_{n}(\log n)^{2}=0 \tag{E2}
\end{equation}

It will be shown that $\epsilon_{n}\left( \log n\right) ^{2}$ is the rate of
accuracy of the approximating scheme.

Denote $a:=a_{n}$ the generic term of the convergent sequence in $\mathbb{R}^{d}$ of $\left(a_{n}\right)_{n\geq 1}.$

Define the approximating density $g_{na}(y_{1}^{k})$ on $\left(\mathbb{R}^{d}\right)^{k}$ as follows.
Denote
\begin{eqnarray}
g_{0}(y_{1}|y_{0}):=\pi^{a}(y_{1})
\end{eqnarray}
where $y_{0}$ is arbitrary and $\pi^{a}$ is defined by (\ref{def:pi_R_d_chap5}).
For $1\leq{i\leq{k-1}}$, define  $g(y_{i+1}|y_{1}^{i})$ recursively.
Let $t_{i} \in \mathbb{R}^{d}$ be the unique solution of the equation
\begin{eqnarray} \label{def:m_i_r_d}
m_{i}:=m(t_{i})=\frac{n}{n-1}\left(a-\frac{s_{1,i}}{n}\right)
\end{eqnarray}
where $s_{1,i}:=y_{1}+...+y_{i}.$

Define
\begin{eqnarray}
\kappa_{(i,n)}^{j,l}:=\frac{d^{2}}{dt^{(j)}dt^{(l)}}\left(\log E_{\pi^{m_{i}}}\exp <t,\mathbf{X}>\right)\left(0\right) 
\end{eqnarray}
and
\begin{eqnarray}
\kappa_{(i,n)}^{j,l,m}:=\frac{d^{3}}{dt^{(j)}dt^{(l)}dt^{(m)}}\left(\log E_{\pi^{m_{i}}}\exp <t,\mathbf{X}>\right)\left(0\right) .
\end{eqnarray}

Let
\begin{eqnarray} \label{def:g_y_i+1_Rd_chap4}
g(y_{i+1}|y_{1}^{i}):=C_{i}\mathfrak{n}_{d}(y_{i+1};\beta\alpha+a,\beta)p(y_{i+1})
\end{eqnarray}
where
\begin{eqnarray} \label{alpha_r_d}
\alpha:=\left(t_{i}+\frac{\kappa_{(i,n)}^{-2}\gamma}{2(n-i-1)}\right)
\end{eqnarray}
\begin{eqnarray} \label{beta_r_d}
\beta:=\kappa_{(i,n)}(n-i-1)
\end{eqnarray}
and $\gamma$ is defined by
\begin{eqnarray}
\gamma:=(\sum_{j=1}^{d} \kappa_{(i,n)}^{j,j,p})_{1\leq{p}\leq{d}}.
\end{eqnarray}
$C_{i}$ is the normalizing constant allowing $g(y_{i+1}|y_{1}^{i})$ to be a density.

Finally, define
\begin{eqnarray} \label{def:g_a_Rd_chap4}
g_{na}(y_{1}^{k}):=g_{0}(y_{1}|y_{0})\prod_{i=1}^{k-1}g(y_{i+1}|y_{1}^{i}).
\end{eqnarray}
We now can state the main theorem.
\begin{theorem} \label{th:princ}
Assume (\ref{cond:A1}) et (\ref{cond:A2}).
\begin{enumerate}
\item Let $Y_{1}^{k}$ be a sample with density $p_{na}.$ Then
\begin{eqnarray} \label{th:cond_ponc_Rd_chap1}
p(\textbf{X}_{1}^{k}=Y_{1}^{k}|\mathbf{S}_{1,n}=na)=g_{na}(Y_{1}^{k})
(1+o_{P_{na}}(1+\epsilon_{n}(\log n)^{2}))
\end{eqnarray}
\item Let $Y_{1}^{k}$ be a sample of density $g_{na}.$ Then
\begin{eqnarray} \label{th:cond_ponc_Rd_chap1_inverse}
p(\textbf{X}_{1}^{k}=Y_{1}^{k}|\mathbf{S}_{1,n}=na)=g_{na}(Y_{1}^{k})
(1+o_{G_{na}}(1+\epsilon_{n}(\log n)^{2}))
\end{eqnarray}
\end{enumerate}
\end{theorem}

\begin{proof}
The proof of this theorem use the same argument of Theorem 2 of \cite{BroniatowskiCaron2013}. The proof uses Bayes formula to write $p(\left.\mathbf{X}_{1}^{k}=Y_{1}^{k}\right\vert \mathbf{S}_{1,n}=na)$ as a product of $k$ conditional densities of individual terms of the trajectory evaluated at $Y_{1}^{k}$. Each term of this product is approximated through an Edgeworth expansion which together with the properties of $Y_{1}^{k}$ under $P_{na}$ (see Lemma \ref{lem:min} and \ref{lem:max}) concludes the proof. By analogy with the real case, this two lemmas allow us to control the terms from the Edgeworth expansion. This lemmas are pertaining to the coordinates of the $m_{i}$ and the $Y_{i+1}.$ This proof is rather long and we have differed its technical steps to the Appendix.

Denote $S_{1,0}:=0,$ $S_{1,1}=Y_{1}$ et $S_{1,i}=S_{1,i-1}+Y_{i}.$ Then

\[p(\textbf{X}_{1}^{k}=Y_{1}^{k}|\mathbf{S}_{1,n}=na )=  \]

\[p(\textbf{X}_{1}=Y_{1}|\mathbf{S}_{1,n}=na ) \prod_{i=1}^{k-1} p(\textbf{X}_{i+1} =Y_{i+1}|\textbf{X}_{1}^{i}=Y_{1}^{i},\mathbf{S}_{1,n}=na)\]

\[ = \prod_{i=0}^{k-1}p( \textbf{X}_{i+1}=Y_{i+1}|\mathbf{S}_{i+1,n}=na -S_{1,i}).\] 

We make use of the following property which states the invariance of
conditional densities under the tilting: For $1\leq i\leq j\leq n,$ for all $%
a$ in the range of $\mathbf{X},$ for all $u$ and $s$ 
\begin{equation}
p\left( \left. \mathbf{S}_{i,j}=u\right\vert \mathbf{S}_{1,n}=s\right) =\pi
^{a}\left( \left. \mathbf{S}_{i,j}=u\right\vert \mathbf{S}_{1,n}=s\right)
\end{equation}%
where $\mathbf{S}_{i,j}:=\mathbf{X}_{i}+...+\mathbf{X}_{j}$ together with $%
\mathbf{S}_{1,0}=s_{1,0}=0$.

Applying this formula, it holds

\[p(\textbf{X}_{i+1}=Y_{i+1}|\mathbf{S}_{i+1,n}=na-S_{1,i})\]

\[=\pi^{m_{i}}(\textbf{X}_{i+1}=Y_{i+1}|\mathbf{S}_{i+1,n}=na-S_{1,i})\]

\[=\pi^{m_{i}}(\textbf{X}_{i+1}=Y_{i+1})\frac{\pi
^{m_{i}}(\mathbf{S}_{i+2,n}=na -S_{1,i+1}) }{\pi ^{m_{i}}(\mathbf{S}_{i+1,n}=na -S_{1,i})}\]

where we used the independence of the $\textbf{X}_{j}$ under $\pi ^{m_{i}}.$ A precise evaluation of each dominating terms in the previous fraction is needed.
Under the sequence of densities $\pi ^{m_{i}}$, the i.i.d. random vectors $\textbf{X}_{i+1},...,\textbf{X}_{n}$ define a triangular array which satisfy a central limit theorem and an Edgeworth expansion. Under $\pi ^{m_{i}}$, $\textbf{X}_{i+1}$ has expectation $m_{i}$ and covariance matrix $\kappa_{(i,n)}.$

Center and normalize both the numerator and denominator in the fraction which
appears in the last display and an Edgeworth expansion to the order
5 is performed for the numerator and the denominator. The main arguments
used in order to obtain the order of magnitude of the involved quantities
are 
\begin{enumerate}
\item  a maximal inequality which controls the magnitude of $m_{i}^{j}$ for
all $i$ between $0$ and $k-1$ and $j$ between $1$ and $d$ stated in Lemma \ref{lem:max}.
\item the order of the maximum of the $Y_{i}^{\prime }s$  stated in Lemma \ref{lem:min}.
\end{enumerate}
As proved in the appendix, it holds under (\ref{cond:A1}) and (\ref{cond:A2}),

\begin{eqnarray}
p(\textbf{X}_{i+1}=Y_{i+1}|\mathbf{S}_{i+1,n}=na-S_{1,i})=\frac{\sqrt{n-i}}{\sqrt{n-i-1}} \pi^{m_{i}}\left(\textbf{X}_{i+1}=Y_{i+1}\right)\frac{N_{i}}{D_{i}}
\end{eqnarray}
where $N_{i}$ and $D_{i}$ are defined by
\begin{eqnarray} \label{def:Ni}
N_{i}:=\exp\{-\frac{{}^t\left(Y_{i+1}-a\right)
\kappa_{(i,n)}^{-1}\left(Y_{i+1}-a\right)}{2(n-i-1)}\}A_{i}+O_{P_{na}}(\frac{1}{(n-i-1)^{3/2}})
\end{eqnarray}
with
\begin{eqnarray} \label{def:Ai}
A_{i}:=1+\frac{{}^{t}\left(Y_{i+1}-a\right)\kappa_{(i,n)}^{-2}\gamma}{2(n-i-1)} +\frac{\delta^{(i,n)}}{n-i-1}+\frac{o_{P_{na}}(\epsilon_{n}(\log n))}{n-i-1}
\end{eqnarray}

and

\begin{eqnarray} \label{def:Di}
D_{i}:=1+\frac{\delta^{(i,n)}}{n-i}+O_{P_{na}}(\frac{1}{(n-i)^{3/2}})
\end{eqnarray}
where the expression of $\delta^{(i,n)}$,defined in (\ref{def:f_total}), depends of the cumulants.

The term $O_{P_{na}}(\frac{1}{(n-i-1)^{3/2}})$ in (\ref{def:Ni}) is uniform in $Y_{i+1}.$ 

The terms in the expression of the approximating density come from an expansion in both ratio. The Gaussian compound is explicit in (\ref{def:Ni}) and the term  $\frac{{}^{t}Y_{i+1}\kappa_{(i,n)}^{-2}\gamma}{2(n-i-1)}$ is the dominating term of $A_{i}.$ The normalizing factor $C_{i}$ in $g(Y_{i+1}|Y_{1}^{i})$ compensate the term $\frac{\sqrt{n-i}}{\Phi(t_{i})\sqrt{n-i-1}}\exp\left( \frac{{}^{t}a\kappa_{(i,n)}^{-2}\gamma}{2(n-i-1)}\right)$ where $\Phi(t_{i})$ come from the term $\pi^{m_{i}}\left(\textbf{X}_{i+1}=Y_{i+1}\right).$ The product of the rest of the terms allow to obtain the convergence rate $1+o_{P_{na}}(1+\epsilon_{n}(\log n)^{2}).$ The technicals details are left in the Appendix. (\ref{th:cond_ponc_Rd_chap1}) has been proved.
\end{proof}

Applications of Theorem \ref{th:princ} in Importance
Sampling procedures and in Statistics require (\ref{th:cond_ponc_Rd_chap1_inverse}). So assume
that $Y_{1}^{k}$ is a random vector generated under $G_{na}$ with density $g_{na}.$ Can we state that $g_{na}\left( Y_{1}^{k}\right) $ is a good
approximation for $p_{na}\left( Y_{1}^{k}\right) $? This holds true. We
state a simple Lemma in this direction.

Let $\mathfrak{R}_{n}$ and $\mathfrak{S}_{n}$ denote two p.m.'s on $\mathbb{R}%
^{n}$ with respective densities $\mathfrak{r}_{n}$ and $\mathfrak{s}_{n}.$

\begin{lemma}
\label{Lemma:commute_from_p_n_to_g_n} Suppose that for some sequence $%
\varepsilon_{n}$ which tends to $0$ as $n$ tends to infinity%
\begin{equation}
\mathfrak{r}_{n}\left( Y_{1}^{n}\right) =\mathfrak{s}_{n}\left(
Y_{1}^{n}\right) \left( 1+o_{\mathfrak{R}_{n}}(\varepsilon_{n})\right)
\label{p_n equiv g_n under p_n}
\end{equation}
as $n$ tends to $\infty.$ Then 
\begin{equation}
\mathfrak{s}_{n}\left( Y_{1}^{n}\right) =\mathfrak{r}_{n}\left(
Y_{1}^{n}\right) \left( 1+o_{\mathfrak{S}_{n}}(\varepsilon_{n})\right) .
\label{g_n equiv p_n under g_n}
\end{equation}
\end{lemma}

\begin{proof}
Denote 
\begin{equation*}
A_{n,\varepsilon_{n}}:=\left\{ y_{1}^{n}:(1-\varepsilon_{n})\mathfrak{s}%
_{n}\left( y_{1}^{n}\right) \leq\mathfrak{r}_{n}\left( y_{1}^{n}\right) \leq%
\mathfrak{s}_{n}\left( y_{1}^{n}\right) (1+\varepsilon_{n})\right\} .
\end{equation*}
It holds for all positive $\delta$%
\begin{equation*}
\lim_{n\rightarrow\infty}\mathfrak{R}_{n}\left(
A_{n,\delta\varepsilon_{n}}\right) =1.
\end{equation*}
Write 
\begin{equation*}
\mathfrak{R}_{n}\left( A_{n,\delta\varepsilon_{n}}\right) =\int \mathbf{1}%
_{A_{n,\delta\varepsilon_{n}}}\left( y_{1}^{n}\right) \frac{\mathfrak{r}%
_{n}\left( y_{1}^{n}\right) }{\mathfrak{s}_{n}(y_{1}^{n})}\mathfrak{s}%
_{n}(y_{1}^{n})dy_{1}^{n}.
\end{equation*}
Since 
\begin{equation*}
\mathfrak{R}_{n}\left( A_{n,\delta\varepsilon_{n}}\right) \leq
(1+\delta\varepsilon_{n})\mathfrak{S}_{n}\left(
A_{n,\delta\varepsilon_{n}}\right)
\end{equation*}
it follows that 
\begin{equation*}
\lim_{n\rightarrow\infty}\mathfrak{S}_{n}\left(
A_{n,\delta\varepsilon_{n}}\right) =1,
\end{equation*}
which proves the claim.
\end{proof}

As a direct by-product of Theorem \ref{th:princ} and Lemma \ref{Lemma:commute_from_p_n_to_g_n} we obtain  (\ref{th:cond_ponc_Rd_chap1_inverse}).

\begin{remark} \label{remark:gaussian_exact}
When the $\mathbf{X}_{i}$'s are i.i.d. Gaussian standard, the result of the approximation theorem are true for $k=n-1$ without the error term. Indeed, it holds $p(\left. \mathbf{X}_{1}^{n-1}=x_{1}^{n-1}\right\vert \mathbf{S}_{1,n}=na)=g_{a}\left( x_{1}^{n-1}\right) $ for all $x_{1}^{n-1}$ in$\left(\mathbb{R}^{d}\right)^{n-1}$. Even in more complicated case, the approximation theorem can be still true. Consider a very simple example. Let $d=s=2$ and $u(x,y)=ax+by$ with $a$ and $b$ known constant. One wants to estimate 
\[
P[\frac{1}{n}\sum_{i=1}^{n}u\left(\mathbf{X}_{i}^{(1)},\mathbf{X}_{i}^{(2)}\right)>c]
\]
with $c$ a known constant and $p_{\mathbf{X}}$ the two-dimension standard normal density.
Using the non-adaptive (or even adaptive) tilting method in this case, we will obtain a sampling density which is the product between a density of $\mathbf{X}_{i}^{(1)}$ and a density of $\mathbf{X}_{i}^{(2)}$ for $i\in(1,...,n).$ However, when looking closely to the conditioning event, we easily see that $\mathbf{X}_{i}^{(1)}$ and $\mathbf{X}_{i}^{(2)}$ are correlated. Indeed, when we calculate exactly the density of $\left(\mathbf{X}_{i}|\frac{1}{n}\sum_{i=1}^{n}u\left(\mathbf{X}_{i}^{(1)},\mathbf{X}_{i}^{(2)}\right)=c\right)$ and its approximate version $g(\mathbf{X}_{i}),$ the result of the approximation theorem are valid for $k=n-1.$
\end{remark}

\begin{remark} \label{lem:Validite__Edg} 
The Edgeworth expansion in the proof of the Theorem \ref{th:cond_ponc_Rd_chap1} are valid under the condition stated in the Theorem 6.4, p.205 of Barndorff-Nielsen et Cox (1990) \cite{BardnoffCox1990} for fixed $a$ and in the Remark 5 of \cite{BroniatowskiCaron2013} when $a=a_{n}$ is a convergent sequence.
\end{remark}

\section{Generalization to a more general conditionning event} \label{sec:extension_chap4}

In the previous section, the conditional density was approximated when the conditionning event writes as $\left(\mathbf{S}_{1,n}=na_{n}\right).$ For practice, we have to extend this result for a more general conditionning event. We still consider $\textbf{X}_{1}^{n}:=(\textbf{X}_{1},...,\textbf{X}_{n})$, a sequence of i.i.d. random vectors in $\mathbb{R}^{d}$ with density $p_{\mathbf{X}}$ and $u$ a measurable function defined from $\mathbb{R}^{d}$ to $\mathbb{R}^{s}$ with both $d,s\geq{1}$. 

Denote
\[\mathbf{U}_{1,n}=\sum_{i=1}^{n}u\left(\mathbf{X}_{i}\right).\]

We assume that $\mathbf{U}:=u\left(\mathbf{X}\right)$ has a density $p_{\mathbf{U}}$ (with p.m. $P_{\mathbf{U}}$) absolutely continuous with respect to Lebesgue measure on $\mathbb{R}^{s}.$ Consider conditioning event of the form
\begin{eqnarray}  \label{def:evt_cond_sur_f}
\left(\mathbf{U}_{1,n}:=u_{1,n}\right)
\end{eqnarray}
with $u_{1,n}/n$ a convergent sequence.
Futhermore, we assume that $u$ is such that the characteristic function of $\mathbf{U}$ belongs to $L^{r}$ for some $r\geq{1}.$

We assume that $\mathbf{U}$ satisfy the Cramer condition, meaning
\[\Phi_{\mathbf{U}}(t):=E[\exp<t,\mathbf{U}>]<\infty,\textsc{\ \ } t\in V(\underline{0})\subset\mathbb{R}^{s}.\]
and denote 
\begin{eqnarray}
m(t):= {}^{t}\nabla \log(\Phi_{\mathbf{U}}(t)),\textsc{\ \ } t\in V(\underline{0})\subset\mathbb{R}^{s}
\end{eqnarray}
and
\begin{eqnarray}
\kappa(t):= {}^{t}\nabla \nabla \log(\Phi_{\mathbf{U}}(t)),\textsc{\ \ } t\in V(\underline{0})\subset\mathbb{R}^{s}.
\end{eqnarray}
as the mean and the covariance matrix of the tilted density defined by
\begin{eqnarray} \label{def:pi_non_centre_U_r_d}
\pi_{\mathbf{U}}^{\alpha}(u) :=\frac{\exp<t,u>}{\Phi_{\mathbf{U}}(t)}p_{\mathbf{U}}(u)
\end{eqnarray}
where $t$ is the unique solution of $m(t)=\alpha$ for $\alpha$ in the convex hull of $P_{\mathbf{U}},$ see \cite{BarndorffNielsen1978_chap_multi}, p132.

We also defined  
\begin{eqnarray} \label{def:pi_non_centre_r_d}
\pi_{u}^{\alpha}(x) :=\frac{\exp<t,u(x)>}{\Phi_{\mathbf{U}}(t)}p_{\mathbf{X}}(x).
\end{eqnarray}

By extension with the case studied in Section \ref{sec:prop_princ_chap4}, we will denote $p_{u_{1,n}}$ the conditional density and $g_{u_{1,n}}$ its approximation.

We now state the general form of the approximating density.
Denote  $m_{0}=u_{1,n}/n$ and
\begin{eqnarray}
g_{0}(y_{1}|y_{0}):=\pi_{u}^{m_{0}}(y_{1})
\end{eqnarray}
with an arbitrary $y_{0}$ and $\pi_{u}^{m_{0}}$ defined in (\ref{def:pi_non_centre_r_d}).

For $1\leq{i\leq{k-1}}$, we recursively define $g(y_{i+1}|y_{1}^{i})$. Let $t_{i}\in\mathbb{R}^{d}$ be the unique solution of the equation
\begin{eqnarray} \label{equation_t_i}
m(t_{i})=m_{i}:=\frac{u_{1,n}-u_{1,i}}{n-i}
\end{eqnarray}
where $u_{1,i}=u(y_{1})+...+u(y_{i}).$

Denote
\begin{eqnarray}
\kappa_{(i,n)}^{j,l}:=\frac{d^{2}}{dt^{(j)}dt^{(l)}}\left(\log E_{\pi_{\mathbf{U}}^{m_{i}}}\exp <t,\mathbf{U}>\right)\left(0\right) 
\end{eqnarray}
and 
\begin{eqnarray}
\kappa_{(i,n)}^{j,l,m}:=\frac{d^{3}}{dt^{(j)}dt^{(l)}dt^{(m)}}\left(\log E_{\pi_{\mathbf{U}}^{m_{i}}}\exp <t,\mathbf{U}>\right)\left(0\right) .
\end{eqnarray}
for $j,l$ and $m$ in $\{1,...,s\}$

Denote
\begin{eqnarray}
g(y_{i+1}|y_{1}^{i}):=C_{i}\mathfrak{n}_{d}\left(u(y_{i+1});\beta\alpha+m_{0},\beta\right)p_{\mathbf{X}}(y_{i+1})
\end{eqnarray}
where $C_{i}$ is a normalizing factor and
\begin{eqnarray}
\alpha:=\left(t_{i}+\frac{\kappa_{(i,n)}^{-2}\gamma}{2(n-i-1)}\right)\\
\beta:=\kappa_{(i,n)}(n-i-1)
\end{eqnarray}
and $\gamma$ defined by 
\begin{eqnarray}
\gamma:=\left(\sum_{j=1}^{s} \kappa_{(i,n)}^{j,j,p}\right)_{1\leq{p}\leq{s}}
\end{eqnarray}

Then
\begin{eqnarray}
g_{u_{1,n}}(y_{1}^{k}):=g_{0}(y_{1}|y_{0})\prod_{i=1}^{k-1} g(y_{i+1}|y_{1}^{i})
\end{eqnarray}

\begin{theorem} \label{th:princ_general}
Assume (\ref{cond:A1}) et (\ref{cond:A2}).
\begin{itemize}
\item Let $Y_{1}^{k}$ a sample of $P_{u_{1,n}}.$
Then
\begin{eqnarray} \label{th:princ_general_1}
p\left(\textbf{X}_{1}^{k}=Y_{1}^{k}|\mathbf{U}_{1,n}=u_{1,n}\right)=g_{u_{1,n}}(Y_{1}^{k})
(1+o_{P_{u_{1,n}}}(1+\epsilon_{n}(\log n)^{2}))
\end{eqnarray}
\item Let $Y_{1}^{k}$ a sample of $G_{u_{1,n}}.$
Then
\begin{eqnarray} \label{th:princ_general_inverse}
p\left(\textbf{X}_{1}^{k}=Y_{1}^{k}|\mathbf{U}_{1,n}=u_{1,n}\right)=g_{u_{1,n}}(Y_{1}^{k})
(1+o_{G_{u_{1,n}}}(1+\epsilon_{n}(\log n)^{2}))
\end{eqnarray}
\end{itemize}
\end{theorem}

\begin{proof}
We only propose the first part of the proof of (\ref{th:princ_general_1}) since the proof's argument of Theorem \ref{th:princ} are used.

Denote: $U_{i,j}:=u(Y_{i})+...+u(Y_{j}).$

Evaluate:
\begin{align*}
p\left( \left. \mathbf{X}_{i+1}=Y_{i+1}\right\vert \mathbf{U}_{i+1,n}=u_{1,n}-U_{1,i}\right) \\
& =p_{\mathbf{X}}\left( \mathbf{X}_{i+1}=Y_{i+1}\right) \frac{p\left(\mathbf{U}_{i+2,n}=u_{1,n}-U_{1,i+1}\right)} {p\left(\mathbf{U}_{i+1,n}=u_{1,n}-U_{1,i}\right)}.
\end{align*}

Multiplying and dividing by $p_{\mathbf{U}}\left(\mathbf{U}_{i+1}=u(Y_{i+1})\right)$ , we use the invariance tilting under $\pi_{\mathbf{U}}^{m_{i}}.$ Then,
\begin{center}
$p\left(\left. \mathbf{X}_{i+1}=Y_{i+1}\right\vert \mathbf{U}_{i+1,n}=u_{1,n}-U_{1,i}\right)$
\end{center}
\begin{eqnarray*}
=\frac{p_{\mathbf{X}}\left(\mathbf{X}_{i+1}=Y_{i+1}\right)}{p_{\mathbf{U}}\left(\mathbf{U}_{i+1}=u(Y_{i+1})\right)} \pi_{\mathbf{U}}^{m_{i}}\left( \mathbf{U}_{i+1}=u(Y_{i+1})\right) \frac{\pi_{\mathbf{U}}^{m_{i}}\left(\mathbf{U}_{i+2,n}=u_{1,n}-U_{1,i+1}\right)} {\pi_{\mathbf{U}}^{m_{i}}\left(\mathbf{U}_{i+1,n}=u_{1,n}-U_{1,i}\right) } \\
\end{eqnarray*}

We proceed to a Edgeworth expansion following the step of the proof of the Theorem \ref{th:princ}. The proof of (\ref{th:princ_general_inverse}) is quite easy using Lemma \ref{Lemma:commute_from_p_n_to_g_n}.
\end{proof}

Now, we can extend our results from typical paths to the whole space $\left(\mathbb{R}^{d}\right)^{k}$. Indeed, convergence of the relative error on large sets imply that the total variation distance between the conditioned measure and its approximation goes to $0$ on the entire space.

\begin{theorem}
Under the hypotheses of Theorem \ref{th:princ_general} the total variation
distance between $P_{u_{1,n}}$ and $G_{u_{1,n}}$ goes to $0$ as $n$ tends to
infinity, and
\begin{equation*}
\lim_{n\rightarrow \infty }\int \left\vert p_{u_{1,n}}\left(
y_{1}^{k}\right) -g_{u_{1,n}}\left( y_{1}^{k}\right) \right\vert
dy_{1}^{k}=0.
\end{equation*}
\end{theorem}

\begin{proof}
See \cite{BroniatowskiCaron2013} for details.
\end{proof}

\section{How far is the approximation valid?} \label{sec:choi_k_chap4}

This section provides a rule leading to an effective choice of the crucial
parameter $k$ in order to achieve a given accuracy bound for the relative
error in Theorem \ref{th:princ}. In \cite{BroniatowskiCaron2013}, a effective rule for assessing the parameter $k$ has been proposed. We adapt this rule to the multivariate case under consideration. This rule is based on an asymptotic expansion which can be found in \cite{Jensen1995_chap_multi}, Chap.6, p.144, Formula (6.1.2). We state a multivariate version of this lemma and extends accordingly the new rule.

The accuracy of the approximation is measured through 
\begin{equation} \label{ERE_GD}
ERE(k):=E_{G_{u_{1,n}}}1_{D_{k}}\left( Y_{1}^{k}\right) \frac{%
p_{u_{1,n}}\left( Y_{1}^{k}\right) -g_{u_{1,n}}\left( Y_{1}^{k}\right) }{%
p_{u_{1,n}}\left( Y_{1}^{k}\right) } 
\end{equation}%
and 
\begin{equation}\label{VRE_GD}
VRE(k):=Var_{G_{u_{1,n}}}1_{D_{k}}\left( Y_{1}^{k}\right) \frac{%
p_{u_{1,n}}\left( Y_{1}^{k}\right) -g_{u_{1,n}}\left( Y_{1}^{k}\right) }{%
p_{u_{1,n}}\left( Y_{1}^{k}\right) }
\end{equation}%
respectively the expectation and the variance of the relative error of the
approximating scheme when evaluated on 
\begin{equation*}
D_{k}:=\left\{ y_{1}^{k}\in \mathbb{R}^{k}\text{ such that }\left\vert
g_{u_{1,n}}(y_{1}^{k})/p_{u_{1,n}}\left( y_{1}^{k}\right) -1\right\vert
<\delta _{n}\right\}
\end{equation*}
with $\epsilon _{n}\left( \log n\right) ^{2}/\delta _{n}\rightarrow 0$ and $%
\delta _{n}\rightarrow 0;$ therefore $G_{u_{1,n}}\left( D_{k}\right)
\rightarrow 1.$ The r.v.'s $Y_{1}^{k}$ are sampled under $%
g_{u_{1,n}}.$ Note that the density $p_{u_{1,n}}$ is usually unknown. The
argument is somehow heuristic and informal; nevertheless the rule is simple
to implement and provides good results. We assume that the set $D_{k}$ can
be substituted by $\mathbb{R}^{k}$ in the above formulas, therefore assuming
that the relative error has bounded variance, which would require quite a
lot of work to be proved under appropriate conditions, but which seems to
hold, at least in all cases considered by the author. We keep the above
notation omitting therefore any reference to $D_{k}.$

Consider a two-sigma confidence bound for the relative accuracy for a given $%
k$, defining 
\begin{equation}\label{CI_GD}
CI(k):=\left[ ERE(k)-2\sqrt{VRE(k)},ERE(k)+2\sqrt{VRE(k)}\right] .
\end{equation}

Let $\delta $ denote an acceptance level for the relative accuracy. Accept $%
k $ until $\delta $ belongs to $CI(k).$ For such $k$ the relative accuracy
is certified up to the level $5\%$ roughly.

The calculation of $VRE(k)$ and $ERE(k)$ should be carried out as follows.

We writes
\begin{align*}
VRE(k)^{2}& =E_{P_{\mathbf{X}}}\left(\frac{g_{u_{1,n}}^{3}\left(Y_{1}^{k}\right)}{p_{u_{1,n}}\left( Y_{1}^{k}\right) ^{2}p_{\mathbf{X}}\left(Y_{1}^{k}\right)}\right) \\
& -E_{P_{\mathbf{X}}}\left(\frac{g_{u_{1,n}}^{2}\left(Y_{1}^{k}\right)} {p_{u_{1,n}}\left(Y_{1}^{k}\right)p_{\mathbf{X}}\left( Y_{1}^{k}\right)}\right)^{2} \\
& =:A-B^{2}.
\end{align*}
By the Bayes formula,
\begin{equation} \label{Jensen_dans_k_limite_chap_multi}
p_{u_{1,n}}\left(Y_{1}^{k}\right)=p_{\mathbf{X}}\left(Y_{1}^{k}\right) \frac{np\left(\mathbf{U}_{k+1,n}/(n-k)=m_{k}\right)} {\left(n-k\right)p\left(\mathbf{U}_{1,n}/n=m_{0}\right)}. 
\end{equation}
with $m_{k}=m(t_{k})$ defined in (\ref{def:m_i_r_d}) and $m_{0}=u_{1,n}/n.$

\begin{lemma} (\cite{Jensen1995_chap_multi}, Chap.6, p.144, Formula (6.1.2)) \label{Lemma:Jensen_chap5} 
Let $\mathbf{U}_{1},...,\mathbf{U}_{n}$ a sampling of $n$ random variables i.i.d. with density $p_{\mathbf{U}}$ on $\mathbb{R}^{d}$ such as the moment generating function exists. Then, with $\nabla\left(\log\Phi_{\mathbf{U}}\right)(t)=u$ and $\Sigma(t):=^{t}\nabla\nabla\left(\log\Phi_{\mathbf{U}}\right)(t)$, it holds
\begin{equation*}
p_{\mathbf{U}_{1}^{n}/n}\left( u\right) =\frac{n^{d/2}\Phi_{\mathbf{U}}^{n}(t)\exp -n<t,u>} {\vert\Sigma(t)\vert^{1/2}\left(2\pi\right)^{d/2}}\left( 1+o(1)\right)
\end{equation*}
when $|u|$ is bounded.
\end{lemma}

Let $D$ and $N$ defined by respectively
\begin{equation*}
D:=\left[ \frac{\pi_{\mathbf{U}}^{m_{0}}(m_{0})}{p_{\mathbf{U}}(m_{0})}\right]^{n}
\end{equation*}
and
\begin{equation*}
N:=\left[\frac{\pi_{\mathbf{U}}^{m_{k}}\left(m_{k}\right)}{p_{\mathbf{U}}\left(m_{k}\right)}\right] ^{\left( n-k\right)}.
\end{equation*}
By (\ref{Jensen_dans_k_limite_chap_multi}) and Lemma \ref{Lemma:Jensen_chap5}, we have
\begin{equation*}
p_{u_{1,n}}\left(Y_{1}^{k}\right)=\left(\frac{n-k}{n}\right)^{\frac{d-2}{2}} p_{\mathbf{X}}\left(Y_{1}^{k}\right) \frac{D}{N}\frac{\vert\Sigma(t)\vert^{1/2}}{\vert\Sigma(t_{k})\vert^{1/2}}\left( 1+o_{p}(1)\right) .
\end{equation*}
Define
\begin{equation*}
A\left( Y_{1}^{k}\right):=\left(\frac{n}{n-k}\right)^{d-2} \left(\frac{g_{u_{1,n}}\left(Y_{1}^{k}\right)} {p_{\mathbf{X}}\left(Y_{1}^{k}\right)}\right)^{3}\left(\frac{N}{D}\right)^{2} \frac{\vert\Sigma(t)\vert}{\vert\Sigma(t_{k})\vert}
\end{equation*}
and simulate $L$ i.i.d. samples $Y_{1}^{k}(l)$, each one made of $k$ i.i.d. replications under $p_{\mathbf{X}}.$ The approximation of $A$ is obtained through Monte Carlo simulation:
\begin{equation*}
\widehat{A}:=\frac{1}{L}\sum_{l=1}^{L}A\left( Y_{1}^{k}(l)\right) .
\end{equation*}
Using the same approximation for $B$, define
\begin{equation*} 
B\left(Y_{1}^{k}\right):=\left(\frac{n}{n-k}\right)^{\frac{d-2}{2}} \left(\frac{g_{u_{1,n}}\left(Y_{1}^{k}\right)} {p_{\mathbf{X}}\left(Y_{1}^{k}\right)}\right)^{2}\left(\frac{N}{D}\right) \frac{\vert\Sigma(t)\vert^{1/2}}{\vert\Sigma(t_{k})\vert^{1/2}}
\end{equation*}
and
\begin{equation*}
\widehat{B}:=\frac{1}{L}\sum_{l=1}^{L}B\left( Y_{1}^{k}(l)\right) 
\end{equation*}
with the same $Y_{1}^{k}(l)$'s as above.

Set
\begin{equation*}
\overline{VRE}(k):=\widehat{A}-\left(\widehat{B}\right)^{2} 
\end{equation*}
\begin{equation*}
\overline{ERE}(k):=1-\widehat{B} 
\end{equation*}
\begin{equation*}
\overline{CI}(k):=\left[ \overline{ERE}(k)-2\sqrt{\overline{VRE}(k)}, \overline{ERE}(k)+2\sqrt{\overline{VRE}(k)}\right].
\end{equation*}

\section{Implementation}

As explained in the introduction, two main applications can be implemented using the approximation proved in this paper. The first one pertains to Importance Sampling scheme and the second one to conditional inference. For this purpose, the implementation of this method can be tricky. Most of the algorithms presented in \cite{BroniatowskiCaron2013} for the real case are still valid. However, the two major difficulties discussed in the real case are still as important.

First, to implement this approximation, we have to solve equation (\ref{equation_t_i}) in $t_{i}$ at each step of the recursive construction. Even in the real case (for example, considering Weibull distribution), the inverse function of $m$ has not a analytic expression and numerical methods have to be considered. 

Secondly,  the simulation of a sample $X_{1}^{k}$ with $g_{u_{1,n}}$ can be fast and easy when $\lim_{n\rightarrow \infty }u_{1,n}/n=Eu\left( \mathbf{X}\right)$. 
Indeed the r.v. $\mathbf{X}_{i+1}$ with density $g\left(
x_{i+1}|x_{1}^{i}\right) $ is obtained through a standard 
acceptance-rejection algorithm. This is in contrast with the case when the
conditioning value is in the range of a large deviation event, i.e. $\lim_{n\rightarrow \infty }u_{1,n}/n\neq Eu\left( \mathbf{X}\right),$ which
appears in a natural way in Importance sampling estimation for rare event
probabilities; then MCMC techniques can be used.

\section{Conclusion}

This paper extends the results of \cite{BroniatowskiCaron2013}. In future work, the author will focus on developing two mains applications: Importance Sampling scheme for multi-constraints probabilities and conditional inference in exponential curved family.

\appendix

\section{Proof of Lemmas}

The next three lemmas are similar to the Lemmas 21, 22 and 23 of \cite{BroniatowskiCaron2013}. This lemmas are stated for the coordinates of each random variable, so the proof are exactly the same as the one stated in the univariate case. However, the explicit proof can be found in \cite{Caron2012}, p.118 and followings.

\begin{lemma} \label{lemme:moments}
For all $j,p,q$ in $\{1,...,d\}$, it holds
\begin{enumerate}
\item $E_{P_{na}}\left( \mathbf{X}_{1}^{(j)}\right) =a^{(j)}$,
\item $E_{P_{na}}\left( \mathbf{X}_{1}^{(p)}\mathbf{X}_{2}^{(q)}\right) =a^{(p)}a^{(q)}+0\left( \frac{1}{n}\right)$
\item $E_{P_{na}}\left( \mathbf{X}_{1}^{(p)}\mathbf{X}_{1}^{(q)}\right) =\kappa_{p,q}(t)+a^{(p)}a^{(q)}+0\left( \frac{1}{n}\right) $
\end{enumerate} 
where $\kappa_{p,q}(t)=\left(^{t}\nabla\nabla\log\Phi(t)\right)_{p,q}$ and $t$ is such as $m(t)=a.$ 
\end{lemma}

\begin{lemma} \label{lem:min}
Under (\ref{cond:A1}),  and for all $j$ between $1$ and $d$, it holds
\begin{eqnarray}
\max_{1\leq i\leq k}\left\vert m_{i}^{(j)}\right\vert =a^{(j)}+o_{P_{na}}\left( \epsilon_{n}\right) .
\end{eqnarray}
It also holds, for all $j,r,s$ in $\{1,...,d\}$, $\max_{1\leq i\leq k}\left\vert \kappa_{(i,n)}^{j,r}\right\vert$ and $\max_{1\leq i\leq k}\left\vert \kappa_{(i,n)}^{j,r,s}\right\vert$ tend in $P_{na}$ probability respectively, to $(\sigma_{j,r})$, the generic term of the covariance matrix and to $(\kappa^{j,r,s})$ the joint cumulant of of $\pi^{\bar{a}}$ where $\bar{a}=\lim_{n\to\infty}a_{n}=\bar{a}.$
\end{lemma}

\begin{lemma} \label{lem:max}
For all $j$ between $1$ and $d$, it holds
\begin{eqnarray}
\max\left(\vert\mathbf{X}_{1}^{(j)}\vert,...,\vert\mathbf{X}_{n}^{(j)}\vert\right)=O_{P_{na}}(\log n)
\end{eqnarray}
\end{lemma}

\begin{lemma} \label{lem:maxcent}
Denote $\mathbf{V}_{i+1}:=\kappa_{(i,n)}^{-1/2}(\mathbf{X}_{i+1}-m_{i,n}).$ Then,
\begin{eqnarray}
\forall j \in \{1,...,d\} \textsc{\ \ } \max\left(\vert\mathbf{V}_{1}^{(j)}\vert,...,\vert\mathbf{V}_{n}^{(j)}\vert\right)=O_{P_{na}}(\log n)
\end{eqnarray}
\end{lemma}

\begin{proof}
Let $j \in \{1,...,d\}$ and $i \in \{1,...,n\}$
\begin{align*}
|\mathbf{V}_{i}^{(j)}|=|[\kappa_{(i,n)}^{-1/2}(\mathbf{X}_{i+1}-m_{i,n})]^{(j)}|\\
&\leq{\displaystyle{\sup_{1\leq{l\leq{d}}}|\alpha_{j,l}|} \displaystyle{\sup_{1\leq{l\leq{d}}}|\mathbf{X}_{i+1}^{(l)}-m_{i,n}^{(l)}|}}\\
&\leq{\sup_{1\leq{l\leq{d}}}|\alpha_{j,l}|(\sup_{1\leq{l\leq{d}}}|\mathbf{X}_{i+1}^{(l)}| +\sup_{1\leq{l\leq{d}}}|m_{i,n}^{(l)}|)}\\
&\leq{\sup_{1\leq{l\leq{d}}}|\alpha_{j,l}|(O_{P_{na}}(\log n)+a^{(j)}+o_{P_{na}}\left( \epsilon_{n}\right))}\\
&\leq{C O_{P_{na}}(\log n)}
\end{align*}
using Lemma \ref{lem:min} and \ref{lem:max} and, for simplicity, denoting $\left(\kappa_{(i,n)}^{-1/2}\right)_{1\leq{j,l\leq{d}}}:=(\alpha_{j,l})_{1\leq{j,l\leq{d}}}.$

Then
\begin{eqnarray*}
\max_{1\leq{i\leq{k-1}}}\vert\mathbf{V}_{i}^{(j)}\vert \leq{\sup_{1\leq{l\leq{d}}}|\alpha_{j,l}| (\max_{1\leq{i\leq{k-1}}}\sup_{1\leq{l\leq{d}}}|\mathbf{X}_{i+1}^{(l)}| +\max_{1\leq{i\leq{k-1}}}\sup_{1\leq{l\leq{d}}}|m_{i,n}^{(l)}|)}\\
\end{eqnarray*}
Finally, it holds
\begin{eqnarray*}
\max_{1\leq{i\leq{k-1}}}\vert\mathbf{V}_{i}^{(j)}\vert\leq{C O_{P_{na}}(\log n)}
\end{eqnarray*}
\end{proof}

\section{Proof of Theorem \ref{th:princ}}

This proof follows the same step as the proof of Theorem 2 of \cite{BroniatowskiCaron2013}.

Denote
\begin{eqnarray} \label{def:U}
\mathbf{V}_{i+1}:=\kappa_{(i,n)}^{-1/2}(\textbf{X}_{i+1}-m_{i,n})
\end{eqnarray}
and $\mathbf{V}_{i+2,n}:=\sum_{j=i+2}^{n}\mathbf{V}_{j}.$

Under $\pi ^{m_{i}}$, $\textbf{V}_{i+1}$ is centered and has for covariance matrix $I_{d}$.

Denote $\overline{\pi_{n-i-1}^{m_{i}}}$ the density of the partial normalized sum $\textbf{V}_{i+2,n}/(\sqrt{n-i-1})$ when the random vectors are i.i.d. with density $\pi ^{m_{i}}.$ We evaluate $\overline{\pi _{n-i-1}^{m_{i}}}$ at $U_{i+1}=\kappa_{(i,n)}^{-1/2}(Y_{i+1}-m_{i})$

\[p(\textbf{X}_{i+1}=Y_{i+1}\vert\textbf{S}_{i+1,n}=na -S_{1,i})\]
\[=\frac{\sqrt{n-i}}{\sqrt{n-i-1}}\pi ^{m_{i}}(\textbf{X}_{i+1}=Y_{i+1}) \frac{\overline{\pi _{n-i-1}^{m_{i}}}(-U_{i+1}/\sqrt{n-i-1})}{\overline{\pi
_{n-i}^{m_{i,n}}}(0)}.\]

Under the hypothesis of Section \ref{sec:intro_chap4}, an Edgeworth can be performed for the numerator and the denominator.
Denote $Z_{i+1}:=-U_{i+1}/\sqrt{n-i-1}.$ 

\subsubsection{Edgeworth expansion}

We begin by the numerator.

\[\overline{\pi _{n-i-1}^{m_{i,n}}}(Z_{i+1})=\mathfrak{n}_{d}(Z_{i+1})
[ 
\begin{array}{c}
1+\frac{1}{\sqrt{n-i-1}}Q_{3}(Z_{i+1})+\frac{1}{n-i-1}Q_{4}(Z_{i+1}) 
\end{array}]
+O_{P_{na}}( \frac{Q_{4}(Z_{i+1}) }{( n-i-1) ^{3/2}}) 
\]

uniformly in $Z_{i+1}$ with $Q_{3}$ and $Q_{4}$ defined in Section \ref{sec:edg} by 
\begin{eqnarray*}
Q_{3}(x):=\frac{1}{6} \kappa^{j,l,m} h_{jlm}(x)
\end{eqnarray*}
and
\begin{eqnarray*}
Q_{4}(x):=\frac{1}{24} \kappa^{j,l,m,q}h_{jlmq}(x) + \frac{1}{72}\kappa^{j,l,m}\kappa^{q,r,s}h_{jlmqrs}(x).
\end{eqnarray*}

We want to obtain a polynomial expansion in terms of power $(n-i)$. The cumulants in the Hermite tensoriel moment are the cumulants of $U$. At the end of this section, when we will turn back to the $X_{i}$'s, we will have to be careful about this cumulants. 

\paragraph*{Study of $Q_{3}$.} It holds

\[\frac{Q_{3}(Z_{i+1})}{\sqrt{n-i-1}}
=\frac{1}{6}\sum_{j=1}^{d}\sum_{l=1}^{d}\sum_{m=1}^{d}
\kappa_{(i,n)}^{j,l,m}(\frac{U_{i+1}^{(j)}U_{i+1}^{(l)}U_{i+1}^{(m)}}{(n-i-1)^{2}}+
\frac{\kappa_{j,l}^{(i,n)}U_{i+1}^{(m)}}{n-i-1}[3]).\]

Using Lemma \ref{lem:max},

\[\frac{Q_{3}(Z_{i+1})}{\sqrt{n-i-1}}
=\frac{1}{6}\sum_{j=1}^{d}\sum_{l=1}^{d}\sum_{m=1}^{d}
\kappa_{(i,n)}^{j,l,m}\frac{\kappa_{j,l}^{(i,n)}U_{i+1}^{(m)}}{n-i-1}[3]
+\frac{O_{P_{na}}((\log n)^{3})}{(n-i-1)^{2}}.\]

Using Section \ref{sec:moments}, we have

\[\kappa_{j,l}^{(i,n)}U_{i+1}^{(m)}[3]=\kappa_{j,l}^{(i,n)}U_{i+1}^{(m)}
+\kappa_{j,m}^{(i,n)}U_{i+1}^{(l)}+\kappa_{l,m}^{(i,n)}U_{i+1}^{(j)}\]

The covariance matrix of $U_{i+1}$ is the identity matrix. So, using the invariance of the cumulants by indice permutations and denoting

\begin{eqnarray} \label{def:gamma}
\gamma=(\sum_{j=1}^{d}\kappa_{(i,n)}^{j,j,m})_{1\leq{m}\leq{d}},
\end{eqnarray}
it holds

\[\frac{Q_{3}(Z_{i+1})}{\sqrt{n-i-1}}=\frac{1}{2(n-i-1)} ^{t}U_{i+1}\gamma+\frac{O_{P_{na}}((\log n)^{3})}{(n-i-1)^{2}}.\]

\paragraph*{Study of $Q_{4}$.} Let split this study in two parts. Denote

\begin{eqnarray*}
A_{1}:=\kappa_{(i,n)}^{j,l,m,p}h_{jlmp}(Z_{i+1})\\
A_{2}:=\kappa_{(i,n)}^{j,l,m}\kappa_{(i,n)}^{p,q,r}h_{jlmpqr}(Z_{i+1})
\end{eqnarray*}

With this notations, $Q_{4}$ can be rewritten
\begin{eqnarray}
\frac{Q_{4}(x)}{n-i-1}=\frac{A_{1}}{24(n-i-1)}+\frac{A_{2}}{72(n-i-1)}
\end{eqnarray}

Using the notations of section \ref{sec:herm},

\[h_{jlmp}(Z_{i+1})=\frac{U_{i+1}^{(j)}U_{i+1}^{(l)}U_{i+1}^{(m)}U_{i+1}^{(p)}}
{(n-i-1)^{2}}-\kappa_{j,l}^{(i,n)}\frac{U_{i+1}^{(m)}U_{i+1}^{(p)}[6]}{n-i-1}
+\kappa_{j,p}^{(i,n)}\kappa_{m,p}^{(i,n)}[3]\]

Using Lemma \ref{lem:max}, \[\frac{U_{i+1}^{(j)}U_{i+1}^{(l)}U_{i+1}^{(m)}U_{i+1}^{(p)}}{(n-i-1)^{3}}=\frac{O_{P_{na}}((\log n)^{4})}{(n-i-1)^{3}}\]

and
\[\kappa_{j,l}\frac{U_{i+1}^{(m)}U_{i+1}^{(p)}[6]}{(n-i-1)^{2}}=\frac{O_{P_{na}}((\log n)^{2})}{(n-i-1)^{2}}\]

Finally,

\[\frac{A_{1}}{24(n-i-1)}= \frac{\delta_{1}^{(i,n)}}{n-i-1}+\frac{O_{P_{na}}((\log n)^{2})}{(n-i-1)^{2}}\]

with $\delta_{1}^{(i,n)}$ defined as

\begin{eqnarray} \label{def:f1}
\delta_{1}^{(i,n)}:=\frac{1}{8}\sum_{j=1}^{d}\sum_{m=1}^{d}\kappa_{(i,n)}^{j,m}
\end{eqnarray}

Recall that

\[h_{jlmpqr}(Z_{i+1})=\frac{U_{i+1}^{(j)}U_{i+1}^{(l)}U_{i+1}^{(m)}U_{i+1}^{(p)}
U_{i+1}^{(q)}U_{i+1}^{(r)}}{(n-i-1)^{3}}-\frac{\kappa_{j,l}^{(i,n)}U_{i+1}^{(m)}
U_{i+1}^{(p)}U_{i+1}^{(q)}U_{i+1}^{(r)}[15]}{(n-i-1)^2}\]
\[+\frac{\kappa_{j,l}^{(i,n)}\kappa_{m,p}^{(i,n)}U_{i+1}^{(q)}U_{i+1}^{(r)}[45]}
{n-i-1}-\kappa_{j,l}^{(i,n)}\kappa_{m,p}^{(i,n)}\kappa_{q,r}^{(i,n)}[15]\]

Using once again Lemma \ref{lem:max}, it holds

\[\frac{U_{i+1}^{(j)}U_{i+1}^{(l)}U_{i+1}^{(m)}U_{i+1}^{(p)}
U_{i+1}^{(q)}U_{i+1}^{(r)}}{(n-i-1)^{4}}=\frac{O_{P_{na}}((\log n)^{6})}{(n-i-1)^{4}},\]

\[\frac{\kappa_{j,l}^{(i,n)}U_{i+1}^{(m)}
U_{i+1}^{(p)}U_{i+1}^{(q)}U_{i+1}^{(r)}[15]}{(n-i-2)^3}=\frac{O_{P_{na}}((\log n)^{4})}{(n-i-1)^{3}}\]
and
\[\frac{\kappa_{j,l}^{(i,n)}\kappa_{m,p}^{(i,n)}U_{i+1}^{(q)}
U_{i+1}^{(r)}[45]}{(n-i-1)^2}=\frac{O_{P_{na}}((\log n)^{2})}{(n-i-1)^{2}}.\]

Finally, 

\[\frac{A_{2}}{72(n-i-1)}=-\frac{\delta_{2}^{(i,n)}}{n-i-1}+\frac{O_{P_{na}}((\log n)^{2})}{(n-i-1)^{2}}\]
with $\delta_{2}^{(i,n)}$ defined as
\begin{eqnarray} \label{def:f2}
\delta_{2}^{(i,n)}:=\frac{15}{72}\sum_{j=1}^{d}\sum_{m=1}^{d}\sum_{q=1}^{d}
\kappa_{(i,n)}^{j,j,m}\kappa_{(i,n)}^{m,q,q}
\end{eqnarray}

\subsubsection{Conclusion of the Edgeworth expansion}

The expansion of $\overline{\pi _{n-i-1}^{m_{i}}}$ writes under (\ref{cond:A1}) and (\ref{cond:A2}).:

\[\overline{\pi_{n-i-1}^{m_{i}}}(Z_{i+1})=\mathfrak{n}_{d}(Z_{i+1})
\left(1+\frac{1}{2(n-i-1)}^{t}U_{i+1}\gamma + \frac{\delta_{1}^{(i,n)}-\delta_{2}^{(i,n)}}{n-i-1}+\frac{O_{P_{na}}((\log n)^{3})}{(n-i-1)^{2}}\right)\]

\[+O_{P_{na}}(\frac{1}{(n-i-1)^{3/2}})\]

We also have the same kind of expansion for the denominator
\begin{eqnarray}
\overline{\pi_{n-i}^{m_{i}}}(0)=\mathfrak{n}_{d}(0)\left(1+\frac{\delta_{1}^{(i,n)}-\delta_{2}^{(i,n)}}{n-i}\right)+
O_{P_{na}}(\frac{1}{(n-i)^{3/2}})
\end{eqnarray}
Then
\begin{eqnarray}
\frac{\overline{\pi _{n-i-1}^{m_{i}}}(Z_{i+1})}{\overline{\pi _{n-i}^{m_{i}}}(0)}=\frac{\mathfrak{n}_{d}(Z_{i+1})}{\mathfrak{n}_{d}(0)}
\frac{1+\frac{1}{2(n-i-1)}{}^{t}U_{i+1}\gamma+\frac{\delta_{1}^{(i,n)}-\delta_{2}^{(i,n)}}{n-i-1}+\frac{O_{P_{na}}((\log n)^{3})}{(n-i-1)^{2}}}{1+\frac{\delta_{1}^{(i,n)}-\delta_{2}^{(i,n)}}{n-i}+O_{P_{na}}(\frac{1}{(n-i)^{3/2}})}
\end{eqnarray}

As mentioned before, we now have to substitute $U_{i+1}$ by $Y_{i+1}$, including in $\gamma$, $\delta_{1}^{(i,n)}$ and $\delta_{2}^{(i,n)}$ which depend implicitly of $U_{i+1}.$ For the last two $\delta_{1}^{(i,n)}$ and $\delta_{2}^{(i,n)}$, there will not be a discussion about the substitution since this terms are not dominant.

Using formula (\ref{def:U}) and (\ref{def:gamma}),

\[^{t}U_{i+1}\gamma=^{t}(\kappa_{(i,n)}^{-1/2}(Y_{i+1}-m_{i,n}))\gamma\]

We state two classical lemmas which can be found in \cite{Hefferon2010}, for example.

\begin{lemma}\label{def:matracine}
An $d\times d$ real positive-definite matrix have $2^{n}$ square root, all symmetric real with only one positive-definite.
\end{lemma} 

\begin{lemma} \label{def:matinv}
All definite-positive matrix are invertible and their inverse are also invertible. If this matrix is symmetric, its inverse is also invertible.
\end{lemma}

The matrix $\kappa_{(i,n)}$ is an $d\times d$ real positive-definite matrix as covariance matrix. So, using Lemma \ref{def:matinv}, $\kappa_{(i,n)}^{-1}$ is symmetric positive-defined and, using Lemma \ref{def:matracine}, $\kappa_{(i,n)}^{-1/2}$ is also symmetric. We choose the only positive-defined. Then $^{t}U_{i+1}\gamma$ rewrites

\[^{t}U_{i+1}\gamma=^{t}(Y_{i+1}-m_{i})\kappa_{(i,n)}^{-1/2}\gamma\]

with $\gamma=\kappa_{(i,n)}^{-3/2}\gamma_{Y}.$

Then
\[\frac{^{t}U_{i+1}\gamma}{2(n-i-1)}=\frac{^{t}Y_{i+1}\kappa_{(i,n)}^{-2}\gamma_{Y}}
{2(n-i-1)}-\frac{^{t}m_{i}\kappa_{(i,n)}^{-2}\gamma_{Y}}{2(n-i-1)}\]

Using Lemma \ref{lem:min} and Lemma \ref{lem:max}, it holds
\begin{eqnarray}
\frac{^{t}U_{i+1}\gamma}{2(n-i-1)}=\frac{^{t}Y_{i+1}\kappa_{(i,n)}^{-2}\gamma_{Y}}
{2(n-i-1)}-\frac{^{t}a\kappa_{(i,n)}^{-2}\gamma_{Y}}{2(n-i-1)}+
\frac{o_{P_{na}}(\epsilon_{n})}{n-i-1}.
\end{eqnarray}

Finally, 

\begin{eqnarray} \label{def:some}
\frac{\overline{\pi _{n-i-1}^{m_{i,n}}}(Z_{i+1})}{\overline{\pi _{n-i}^{m_{i,n}}}(0)}=\frac{\mathfrak{n}_{d}(Z_{i+1})}{\mathfrak{n}_{d}(0)}
\end{eqnarray}
\begin{eqnarray*}
\frac{1+\frac{{}^{t}Y_{i+1}\kappa_{(i,n)}^{-2}\gamma_{Y}}{2(n-i-1)}-
\frac{^{t}a_{n}\kappa_{(i,n)}^{-2}\gamma_{Y}}{2(n-i-1)}+
\frac{o_{P_{na}}(\epsilon_{n})}{n-i-1}+
\frac{\delta_{1}^{(i,n)}-\delta_{2}^{(i,n)}}{n-i-1}+\frac{O_{P_{na}}((\log n)^{3})}{(n-i-1)^{2}}}
{1+\frac{\delta_{1}^{(i,n)}-\delta_{2}^{(i,n)}}{n-i}+O_{P_{na}}(\frac{1}{(n-i)^{3/2}})}
\end{eqnarray*}

Denote $C:=\frac{\mathfrak{n}_{d}(Z_{i+1})}{\mathfrak{n}_{d}(0)}$

\subsubsection{Taylor expansion of $C$}

Recall that $Z_{i+1}=-\frac{U_{i+1}}{\sqrt{n-i-1}}$ and $U_{i+1}=\kappa_{(i,n)}^{-1/2}(Y_{i+1}-m_{i})$. 

Using Lemma \ref{lem:min} and Lemma \ref{lem:max} and under (\ref{cond:A1}) and (\ref{cond:A2}), it holds

\begin{eqnarray} \label{TaylorFinal}
\mathfrak{n}_{d}(Z_{i+1})=\mathfrak{n}_{d}(-\frac{\kappa_{(i,n)}^{-1/2}Y_{i+1}}{\sqrt{n-i-1}})(1+\frac{{}^tY_{i+1}\kappa_{(i,n)}^{-1}a}{n-i-1}
-\frac{{}^ta\kappa_{(i,n)}^{-1}a}{2(n-i-1)}+\frac{o_{P_{na}}(\epsilon_{n}(\log n))}{n-i-1})
\end{eqnarray}

\subsubsection{Final result}

To simplify the notations, denote $\gamma=\gamma_{Y}$ and
\begin{eqnarray}\label{def:f_total}
\delta^{(i,n)}=\delta_{1}^{(i,n)}-\delta_{2}^{(i,n)}.
\end{eqnarray}

We put together the two main results (\ref{def:some}) and (\ref{TaylorFinal}). Then under (\ref{cond:A1}) and (\ref{cond:A2}), we get

\[\frac{\overline{\pi _{n-i-1}^{m_{i}}}(Z_{i+1})}{\overline{\pi _{n-i}^{m_{i}}}(0)}=\exp\{-\frac{{}^tY_{i+1}
\kappa_{(i,n)}^{-1}Y_{i+1}}{2(n-i-1)}\}\]

\[\frac{1+\frac{{}^{t}Y_{i+1}\kappa_{(i,n)}^{-2}\gamma}{2(n-i-1)}-
\frac{^{t}a\kappa_{(i,n)}^{-2}\gamma}{2(n-i-1)}+\frac{\delta^{(i,n)}}{n-i-1}+\frac{{}^tY_{i+1}\kappa_{(i,n)}^{-1}a}{n-i-1}
-\frac{{}^ta\kappa_{(i,n)}^{-1}a}{2(n-i-1)}+\frac{o_{P_{na}}(\epsilon_{n}(\log n))}{n-i-1}}
{1+\frac{\delta^{(i,n)}}{n-i}+O_{P_{na}}(\frac{1}{(n-i)^{3/2}})}\]

Denote

\begin{eqnarray}\label{def:u1}
u_{1}=\frac{{}^{t}Y_{i+1}\kappa_{(i,n)}^{-2}\gamma}{2(n-i-1)}-
\frac{^{t}a\kappa_{(i,n)}^{-2}\gamma}{2(n-i-1)}+\frac{\delta^{(i,n)}}{n-i-1}+\frac{{}^tY_{i+1}\kappa_{(i,n)}^{-1}a}{n-i-1}
-\frac{{}^ta\kappa_{(i,n)}^{-1}a}{2(n-i-1)}+\frac{o_{P_{na}}(\epsilon_{n}(\log n))}{n-i-1}
\end{eqnarray}
and
\begin{eqnarray}\label{def:u2}
u_{2}=\frac{\delta^{(i,n)}}{n-i}+O_{P_{na}}(\frac{1}{(n-i)^{3/2}}).
\end{eqnarray}

Making a second order expansion of the numerator and the denominator, it then holds

\[p(\textbf{X}_{i+1}=Y_{i+1}|\textbf{S}_{i+1,n}=na-S_{1,i})=L_{i}C_{i}\]

\[\exp\{{}^tY_{i+1}(t_{i}+\frac{\kappa_{(i,n)}^{-2}\gamma}{2(n-i-1)})\}
\exp\{-\frac{{}^t(Y_{i+1}-a)\kappa_{(i,n)}^{-1}(Y_{i+1}-a)}{2(n-i-1)}\}
\exp\{o_{P_{na}}(\frac{\epsilon_{n}(\log n)}{n-i-1})\}A(i)\]

with
\[A(i):=\frac{\exp\{\frac{u_{1}^{2}}{2}+o(u_{1}^{2})\}}{\exp\{O_{P_{na}}(\frac{1}{(n-i-1)^{2}})+\frac{u_{2}^{2}}{2}+o(u_{2}^{2})\}}
\exp\{\frac{\delta^{(i,n)}}{n-i-1}-\frac{\delta^{(i,n)}}{n-i}\}\]

and
\[
L_{i}:=\frac{\sqrt{n-i}}{\sqrt{n-i-1}}\frac{C_{i}^{-1}}{\Phi(t_{i,n})}\exp\{{}^ta\frac{\kappa_{(i,n)}^{-2}\gamma}{2(n-i-1)}\}
\]

Then, we obtain

\begin{eqnarray}
p(\textbf{X}_{1}^{k}=Y_{1}^{k}|\mathbf{S}_{1,n}=na)
=g_{0}(Y_{1}|Y_{0})\prod_{i=1}^{k-1}g(Y_{i+1}|Y_{1}^{i})\prod_{i=0}^{k-1}A(i)\prod_{i=0}^{k-1}L(i)
\end{eqnarray}

It remains to prove
\begin{eqnarray} \label{Result_Li}
\prod_{i=0}^{k-1}L(i)=1+o_{P_{na}}(\epsilon_{n}(\log n)^{2})
\end{eqnarray}
and
\begin{eqnarray}\label{Result_Ai}
\prod_{i=0}^{k-1}A(i)=1+o_{P_{na}}(\epsilon_{n}(\log n)^{2}).
\end{eqnarray}

\subsubsection{Proof of (\ref{Result_Li})}
Recall the expression of  $C_{i}^{-1}$ :
\begin{eqnarray}
C_{i}^{-1}=\int \exp\{\frac{{}^tx(t_{i}+\kappa_{(i,n)}^{-2}\gamma)}{2(n-i-1)}\}\exp\{-\frac{{}^t(x-a)\kappa_{(i,n)}^{-1}(x-a)}{2(n-i-1)}\}p(x) dx
\end{eqnarray}

Denote
\begin{equation*}
u_{x}:=\frac{{}^tx\kappa_{(i,n)}^{-2}\gamma}{2\left(n-i-1\right) }+\frac{{}^t(x-a)\kappa_{(i,n)}^{-1}(x-a)}{2\left(n-i-1\right)}.
\end{equation*}

Use the classical bounds for $u_{x}$ :
\begin{equation*}
1-u+\frac{u_{x}^{2}}{2}-\frac{u_{x}^{3}}{6}\leq e^{-u_{x}}\leq 1-u_{x}+\frac{u_{x}^{2}}{2}
\end{equation*}
to obtain on both sides of the above inequalities the second order approximation of $C_{i}^{-1}$ through integration with respect to $p.$ The upper bound yields

\begin{align*}
C_{i}^{-1} =\int_{\mathbb{R}^{d}} \exp\{<x,t_{i}>\}\exp\{-u_{x}\}p(x)dx \\
\leq \Phi(t_{i})(1+\int_{\mathbb{R}^{d}}\frac{{}^tx\kappa_{(i,n)}^{-2}\gamma}{2(n-i-1)} \frac{\exp\{<x,t_{i}>\}p(x)}{\Phi(t_{i})}dx\\
-\int_{\mathbb{R}^{d}}\frac{{}^t(x-a)\kappa_{(i,n)}^{-1}(x-a)}{2(n-i-1)} \frac{\exp\{<x,t_{i}>\}p(x)}{\Phi(t_{i})} \\
+O_{P_{na}}\left( \frac{1}{(n-i-1)^{2}}\right))
\end{align*}

where the approximation term is uniform in the $Y_{1}^{k}.$

Using Lemme \ref{lem:min} and the definition of $m_{i}$ and $\kappa_{(i,n)}$, we finally get
\begin{eqnarray}
C_{i}^{-1} \leq{\Phi(t_{i})\left(1+\frac{{}^ta\kappa_{(i,n)}^{-2}\gamma}{2(n-i-1)}-\frac{1}{2(n-i-1)}+\frac{o_{P_{na}}\left(\epsilon_{n}\right)}{n-i-1}\right)} 
\end{eqnarray}

Then,

\begin{eqnarray}
L_{i} \leq{\frac{\sqrt{n-i}}{\sqrt{n-i-1}}\exp\{\frac{{}^ta\kappa_{(i,n)}^{-2}\gamma}{2(n-i-1)}\}
(1+\frac{{}^ta\kappa_{(i,n)}^{-2}\gamma}{2(n-i-1)}-\frac{1}{2(n-i-1)}+\frac{o_{P_{na}}\left(\epsilon_{n}\right)}{n-i-1})}
\end{eqnarray}

Substituting $\frac{\sqrt{n-i}}{\sqrt{n-i-1}}$ and $\exp\left( -\frac {{}^ta\kappa_{(i,n)}^{-2}\gamma}{2(n-i-1)}\right) $ by their expansion $1+\frac{1}{2\left(n-i-1\right) }+O\left( \frac{1}{\left( n-i-1\right) ^{2}}\right) $ and $1-\frac {{}^ta\kappa_{(i,n)}^{-2}\gamma}{2(n-i-1)}+O\left( \frac{||a||^{2}}{(n-i-1)^{2}}\right) $ in the upper bound of $L_{i}$; it then holds

\begin{eqnarray*}
L_{i} \leq\ \left( 1+\frac{1}{2(n-i-1)}+O\left( \frac{1}{\left( n-i-1\right) ^{2}}\right)\right)
\left(1-\frac {{}^ta\kappa_{(i,n)}^{-2}\gamma}{2(n-i-1)}+O\left( \frac{||a||^{2}}{(n-i-1)^{2}}\right)\right) \\
(1+\frac{{}^ta\kappa_{(i,n)}^{-2}\gamma}{2(n-i-1)}-\frac{1}{2(n-i-1)}+\frac{o_{P_{na}}\left(\epsilon_{n}\right)}{n-i-1}) 
\end{eqnarray*}
Writes
\begin{equation*}
\prod_{i=1}^{k}L_{i}\leq{\prod_{i=1}^{k}}\left( {1+M_{i}}\right)
\end{equation*}
with
\begin{equation*}
M_{i}=-\frac{({}^ta\kappa_{(i,n)}^{-2}\gamma)^{2}}{4(n-i-1)^{2}}+\frac{o_{P_{na}}(\epsilon_{n})}{n-i-1}.
\end{equation*}

Under (\ref{cond:A1}) and (\ref{cond:A2}), $\sum_{i=0}^{k-1}M_{i}$ is $o_{P_{na}}\left(\epsilon_{n}\left( \log n\right) ^{2}\right).$

\subsubsection{Proof of (\ref{Result_Ai})}

We make use of the following version of the law of large numbers for
triangular arrays (see \cite{Taylor1985} Theorem 3.1.3).

\begin{theorem} \label{Theo:Taylor}Let $X_{i,n}$ ,$1\leq i\leq k$ denote an array of row-wise real exchangeable r.v.'s and $\lim_{n\rightarrow\infty}k=\infty.$ Let $\rho_{n}:=EX_{1,n}X_{2,n}.$ Assume that for some finite $\Gamma$ , $E[X_{1,n}^{2}]\leq\Gamma.$ If for some doubly indexed sequence $\left(a_{i,n}\right) $ such that $\lim_{n\rightarrow\infty}%
\sum_{i=1}^{k}a_{i,n}^{2}=0$ it holds that
\begin{equation*}
\lim_{n\rightarrow\infty}\rho_{n}\left( \sum_{i=1}^{k}a_{i,n}^{2}\right)
^{2}=0
\end{equation*}
then%
\begin{equation*}
\lim_{n\rightarrow\infty}\sum_{i=1}^{k}a_{i,n}X_{i,n}=0
\end{equation*}
in probability.
\end{theorem}

We want to prove $\prod_{i=0}^{k-1}A(i)=1+o_{P_{na}}(\epsilon_{n}(\log n)^{2})$
where
\[A(i):=\frac{\exp\{\frac{u_{1}^{2}}{2}+o(u_{1}^{2})\}}{\exp\{O_{P_{na}}(\frac{1}{(n-i-1)^{2}})+\frac{u_{2}^{2}}{2}+o(u_{2}^{2})\}}
\exp\{\frac{\delta^{(i,n)}}{n-i-1}-\frac{\delta^{(i,n)}}{n-i}\}\]
with $u_{1}$ and $u_{2}$ defined in (\ref{def:u1}) and in (\ref{def:u2}).

Let $\beta_{j}$ a positive number for $j \in \{1,...,17\}$ and denote

\begin{equation*}
A_{n}^{1}:=\left\{ \frac{1}{\epsilon_{n}\left( \log n\right) ^{2}}\sum_{i=0}^{k-1}\left\vert \frac{\delta^{(i,n)}}{(n-i-1)(n-i)}\right\vert <\beta_{1}\right\}
\end{equation*}

\begin{equation*}
A_{n}^{2}:=\left\{ \frac{1}{\epsilon_{n}\left( \log n\right) ^{2}}\sum_{i=0}^{k-1}\left\vert \frac{(\delta^{(i,n)})^{2}}{(n-i)^{2}}\right\vert <\beta_{2}\right\}
\end{equation*}

\begin{equation*}
A_{n}^{3}:=\left\{ \frac{1}{\epsilon_{n}\left( \log n\right) ^{2}}\sum_{i=0}^{k-1}\left\vert \frac{(\delta^{(i,n)})^{2}}{(n-i-1)^{2}}\right\vert <\beta_{3}\right\}
\end{equation*}

\begin{equation*}
A_{n}^{4}:=\left\{ \frac{1}{\epsilon_{n}\left( \log n\right) ^{2}}\sum_{i=0}^{k-1}\left\vert \frac{(^{t}a\kappa_{(i,n)}^{-2}\gamma)^{2}}{(n-i-1)^{2}}\right\vert <\beta_{4}\right\}
\end{equation*}

\begin{equation*}
A_{n}^{5}:=\left\{ \frac{1}{\epsilon_{n}\left( \log n\right) ^{2}}\sum_{i=0}^{k-1}\left\vert \frac{({}^ta\kappa_{(i,n)}^{-1}a)^{2}}{(n-i-1)^{2}}\right\vert <\beta_{5}\right\}
\end{equation*}

\begin{equation*}
A_{n}^{6}:=\left\{ \frac{1}{\epsilon_{n}\left( \log n\right) ^{2}}\sum_{i=0}^{k-1}\left\vert \frac{\delta^{(i,n)}(^{t}a\kappa_{(i,n)}^{-2}\gamma)}{(n-i-1)^{2}}\right\vert <\beta_{6}\right\}
\end{equation*}

\begin{equation*}
A_{n}^{7}:=\left\{ \frac{1}{\epsilon_{n}\left( \log n\right) ^{2}}\sum_{i=0}^{k-1}\left\vert \frac{\delta^{(i,n)}({}^ta\kappa_{(i,n)}^{-1}a)}{(n-i-1)^{2}}\right\vert <\beta_{7}\right\}
\end{equation*}

\begin{equation*}
A_{n}^{8}:=\left\{ \frac{1}{\epsilon_{n}\left( \log n\right) ^{2}}\sum_{i=0}^{k-1}\left\vert \frac{(^{t}a\kappa_{(i,n)}^{-2}\gamma)({}^ta\kappa_{(i,n)}^{-1}a)}{(n-i-1)^{2}}\right\vert <\beta_{8}\right\}
\end{equation*}

\begin{equation*}
A_{n}^{9}:=\left\{ \frac{1}{\epsilon_{n}\left( \log n\right) ^{2}}\sum_{i=0}^{k-1}\left\vert \frac{({}^{t}Y_{i+1}\kappa_{(i,n)}^{-2}\gamma)^{2}}{(n-i-1)^{2}}\right\vert <\beta_{9}\right\}
\end{equation*}

\begin{equation*}
A_{n}^{10}:=\left\{ \frac{1}{\epsilon_{n}\left( \log n\right) ^{2}}\sum_{i=0}^{k-1}\left\vert \frac{({}^tY_{i+1}\kappa_{(i,n)}^{-1}a)^{2}}{(n-i-1)^{2}}\right\vert <\beta_{10}\right\}
\end{equation*}

\begin{equation*}
A_{n}^{11}:=\left\{ \frac{1}{\epsilon_{n}\left( \log n\right) ^{2}}\sum_{i=0}^{k-1}\left\vert \frac{({}^{t}Y_{i+1}\kappa_{(i,n)}^{-2}\gamma)(^{t}a\kappa_{(i,n)}^{-2}\gamma)}{(n-i-1)^{2}}\right\vert <\beta_{11}\right\}
\end{equation*}

\begin{equation*}
A_{n}^{12}:=\left\{ \frac{1}{\epsilon_{n}\left( \log n\right) ^{2}}\sum_{i=0}^{k-1}\left\vert \frac{\delta^{(i,n)}({}^{t}Y_{i+1}\kappa_{(i,n)}^{-2}\gamma)}{(n-i-1)^{2}}\right\vert <\beta_{12}\right\}
\end{equation*}

\begin{equation*}
A_{n}^{13}:=\left\{ \frac{1}{\epsilon_{n}\left( \log n\right) ^{2}}\sum_{i=0}^{k-1}\left\vert \frac{({}^{t}Y_{i+1}\kappa_{(i,n)}^{-2}\gamma)({}^ta\kappa_{(i,n)}^{-1}a)}{(n-i-1)^{2}}\right\vert <\beta_{13}\right\}
\end{equation*}

\begin{equation*}
A_{n}^{14}:=\left\{ \frac{1}{\epsilon_{n}\left( \log n\right) ^{2}}\sum_{i=0}^{k-1}\left\vert \frac{({}^tY_{i+1}\kappa_{(i,n)}^{-1}a)(^{t}a\kappa_{(i,n)}^{-2}\gamma)}{(n-i-1)^{2}}\right\vert <\beta_{14}\right\}
\end{equation*}

\begin{equation*}
A_{n}^{15}:=\left\{ \frac{1}{\epsilon_{n}\left( \log n\right) ^{2}}\sum_{i=0}^{k-1}\left\vert \frac{\delta^{(i,n)}({}^tY_{i+1}\kappa_{(i,n)}^{-1}a)}{(n-i-1)^{2}}\right\vert <\beta_{15}\right\}
\end{equation*}

\begin{equation*}
A_{n}^{16}:=\left\{ \frac{1}{\epsilon_{n}\left( \log n\right) ^{2}}\sum_{i=0}^{k-1}\left\vert \frac{({}^tY_{i+1}\kappa_{(i,n)}^{-1}a)({}^ta\kappa_{(i,n)}^{-1}a)}{(n-i-1)^{2}}\right\vert <\beta_{16}\right\}
\end{equation*}

\begin{equation*}
A_{n}^{17}:=\left\{ \frac{1}{\epsilon_{n}\left( \log n\right) ^{2}}\sum_{i=0}^{k-1}\left\vert \frac{({}^tY_{i+1}\kappa_{(i,n)}^{-1}a)({}^{t}Y_{i+1}\kappa_{(i,n)}^{-2}\gamma)}{(n-i-1)^{2}}\right\vert <\beta_{17}\right\}
\end{equation*}

We can split this seventeen sets into three groups. For each of this group, we will study an example. The remaining sets can be controlled using the same technics.
For $j \in \{1,...,8\}$, we easily have
\begin{equation*}
\lim_{n\rightarrow\infty}P_{na}\left( A_{n}^{j}\right) =1
\end{equation*}

The second group of sets (from $A_{n}^{9}$ to $A_{n}^{16})$ depends of $Y_{i+1}.$ Consider $A_{n}^{16}.$ For proving $\lim_{n\rightarrow\infty}P_{na}\left( A_{n}^{16}\right) =1$, Theorem \ref{Theo:Taylor} have to be used. The sums in $A_{n}^{16}$ can be rewrite as follows

\begin{align*}
|({}^tY_{i+1}\kappa_{(i,n)}^{-1}a)(^{t}a\kappa_{(i,n)}^{-1}a)|=|<Y_{i+1},\kappa_{(i,n)}^{-1}a><,\kappa_{(i,n)}^{-1}a>| \\
\leq{||Y_{i+1}||||\kappa_{(i,n)}^{-1}a||||a||||\kappa_{(i,n)}^{-1}a||} \\
\leq{\lambda_{1}^{2}||Y_{i+1}||||a||^{3}} \\
\end{align*}
where $\lambda_{1}$ is the largest singular value of $\kappa_{(i,n)}^{-1}.$

We now can apply Theorem \ref{Theo:Taylor} with $X_{i,n}:=||Y_{i+1}||$ and $a_{i,n}:=\frac{||a||^{3}}{\epsilon_{n}\left( \log n\right) ^{2}(n-i-1)^{2}}.$

We check the hypothesis.
\paragraph{1.} Using Lemma \ref{lemme:moments}, it holds
\begin{eqnarray*}
E[X_{1,n}^{2}]=E[||Y_{1}||^{2}]=\sum_{j=1}^{d}E[[Y_{1}^{(j)}]^{2}]=tr(\kappa_{(i,n)})+||a||^{2}.
\end{eqnarray*}
Hence, $E[X_{1,n}^{2}]\leq{\Gamma}$ for some finite $\Gamma.$

\paragraph{2.} It holds
\begin{eqnarray}
\rho_{n}:=E[X_{1,n}X_{2,n}]\leq{\sqrt{E[X_{1,n}^{2}]E[X_{1,n}^{2}]}}=\sqrt{E[||Y_{1}||^{2}]E[||Y_{2}||^{2}]} \notag
\end{eqnarray}
with
\begin{eqnarray}
E[||Y_{1}||^{2}||Y_{2}||^{2}]=E[(\sum_{j=1}^{d}[Y_{1}^{(j)}]^{2})(\sum_{j=1}^{d}[Y_{2}^{(j)}]^{2})]. \notag
\end{eqnarray}

Using Lemma \ref{lemme:moments},
\begin{eqnarray}
\rho_{n}\leq{tr(\kappa_{(i,n)})+||a||^{2}} \notag
\end{eqnarray}

\paragraph{3.} The last two hypothesis are easily checked under (\ref{cond:A1}).
\begin{eqnarray}
\lim_{n \to \infty} \sum_{i=0}^{k-1} a_{i,n}^{2}=\lim_{n \to \infty} \frac{||a||^{6}}{\epsilon_{n}^{2}(\log n)^{4}(n-k)^{3}}=0
\end{eqnarray}
and
\begin{eqnarray}
\lim_{n \to \infty} \rho_{n}\left(\sum_{i=0}^{k-1} a_{i,n}\right)^{2}\leq{\lim_{n \to \infty} \frac{\rho_{n}||a||^{6}}{\epsilon_{n}^{2}(\log n)^{4}(n-k)^{2}}}=0.
\end{eqnarray}

Then, by Theorem \ref{Theo:Taylor}, 
\[\lim_{n\rightarrow\infty}P_{na}\left( A_{n}^{16}\right)=1.\]

Finally consider $A_{n}^{17}$ with

\begin{eqnarray*}
|({}^tY_{i+1}\kappa_{(i,n)}^{-1}a)({}^{t}Y_{i+1}\kappa_{(i,n)}^{-2}\gamma)|\leq{\lambda_{1}^{3}||Y_{i+1}||^{2}||a||}
\end{eqnarray*}
where $\lambda_{1}$ is the largest singular value of $\kappa_{(i,n)}^{-1}.$

We now can apply Theorem \ref{Theo:Taylor} with $X_{i,n}:=||Y_{i+1}||^{2}$ and $a_{i,n}:=\frac{||a||}{\epsilon_{n}\left( \log n\right) ^{2}(n-i-1)^{2}}.$

\paragraph{1}
By Lemma \ref{lemme:moments}, we can find a positive constant $\Gamma$ such as $E[X_{1,n}^{2}]\leq{\Gamma}.$

\paragraph{2}
\begin{eqnarray}
\rho_{n}=E[X_{1,n}X_{2,n}]=E[||Y_{1}||^{2}||Y_{2}||^{2}]=\sum_{j=1}^{d} E[[Y_{1}^{(j)}]^{2}[Y_{2}^{(j)}]^{2}]+\sum_{j=1}^{d} \sum_{l \ne j} E[[Y_{1}^{(j)}]^{2}[Y_{2}^{(l)}]^{2}]
\end{eqnarray}
Using once again Lemma \ref{lemme:moments}, we can show
\begin{eqnarray*}
E[[Y_{1}^{(j)}]^{2}[Y_{2}^{(k)}]^{2}]\leq{\Gamma_{1}}
\end{eqnarray*}

\paragraph{3}. Under, (\ref{cond:A1}),
\begin{eqnarray}
\lim_{n \to \infty} \sum_{i=0}^{k-1} a_{i,n}^{2}=\lim_{n \to \infty} \frac{||a||^{2}}{\epsilon_{n}^{2}(\log n)^{4}(n-k)^{3}}=0
\end{eqnarray}
et
\begin{eqnarray}
\lim_{n \to \infty} \rho_{n}\left(\sum_{i=0}^{k-1} a_{i,n}\right)^{2}\leq{\lim_{n \to \infty} \frac{\rho_{n}||a||^{2}}{\epsilon_{n}^{2}(\log n)^{4}(n-k)^{2}}}=0
\end{eqnarray}

Then, by Theorem \ref{Theo:Taylor}, $\lim_{n\rightarrow\infty}P_{na}\left( A_{n}^{16}\right)=1.$

Denoting $A_{n}:=\bigcup_{j=1}^{17} A_{n}^{j}$, it holds
\[
\lim_{n\rightarrow\infty}P_{na}\left( A_{n}\right)=1.
\]

And $\prod_{i=0}^{k-1}A(i)=1+o_{P_{na}}(\epsilon_{n}(\log n)^{2}).$ This completes the proof.

\end{document}